\documentclass[12pt]{amsart} 

\usepackage[psamsfonts]{amssymb}
\usepackage{times,a4wide}
\usepackage{MnSymbol}

\theoremstyle{plain}

\numberwithin{equation}{section}
\theoremstyle{plain}
\newtheorem{proposition}[equation]{Proposition}
\newtheorem{corollary}[equation]{Corollary}
\newtheorem{theorem}[equation]{Theorem}
\newtheorem{lemma}[equation]{Lemma}
\theoremstyle{definition}

\newtheorem{example}[equation]{Example}
\newtheorem{remark}[equation]{Remark}

\let\stdthebibliography\thebibliography
\let\stdendthebibliography\endthebibliography
\renewenvironment*{thebibliography}[1]{
  \stdthebibliography{BCFMDT10}}
  {\stdendthebibliography}

\newcommand{\nc}{\newcommand}

\newcommand{\N}{{\mathbb N}}
\newcommand{\D}{{\mathbb D}}

\newcommand{\T}{{\mathbb T}} 

\newcommand{\Ker}{\operatorname{Ker}}

\newcommand{\Id}{\operatorname{Id}}
\newcommand{\Rng}{\operatorname{Rng}}

\newcommand{\lra}{\longrightarrow}

\newcommand{\lmto}{\longmapsto}
\newcommand{\eps}{\varepsilon}
\newcommand{\vp}{\varphi}
\nc{\bea}{\begin{eqnarray}}
\nc{\eea}{\end{eqnarray}}
\nc{\beqa}{\begin{eqnarray*}}
\nc{\eeqa}{\end{eqnarray*}}
\nc{\Hi}{H^{\infty}}
\nc{\loi}{\ell^{\infty}}
\nc{\NL}{N^+\vert \Lambda}
\nc{\hf}{{\mathcal H}_{\phi}}
\nc{\liL}{\lambda\in\Lambda}
\nc{\nn}{\nonumber}
\nc{\dst}{\displaystyle}

\newenvironment{proof*}{\vskip 2mm\noindent {}}{$\blacksquare$ \vskip 2mm}
\numberwithin{equation}{section}

\title
{Boundary values in range spaces of co-analytic truncated Toeplitz operators}

\author{Andreas Hartmann \& William T.\ Ross}

\address{Institut de Math\'ematiques de Bordeaux,
Universit\'e Bordeaux I, 351 cours de la Lib\'eration,
33405 Talence, France}

\address{Department of Mathematics and Computer Science, University of Richmond, VA 23173, USA}

\thanks{This work has been done while the first named author was staying at
University of Richmond as the Gaines chair in mathematics. He would like to thank that institution for the hospitality
and support during his stay. This
author is also partially supported by ANR FRAB}

\email{hartmann@math.u-bordeaux.fr, wross@richmond.edu}

\date{\today}



\begin{document}

\bibliographystyle{amsalpha}
\begin{abstract} 
Functions in backward shift invariant subspaces have nice analytic continuation
properties outside the spectrum of the inner function defining the space. Inside the spectrum
of the inner function,
Ahern and Clark showed that under some distribution condition on the zeros and the
singular measure of the inner function, it is possible to obtain non-tangential boundary values
of every function in the backward shift invariant subspace as well as for their
derivatives up to a certain order. Here we will investigate, at least when the inner function is a Blaschke product, the non-tangential boundary values of 
the functions of the backward shift invariant subspace after having applied a co-analytic
(truncated) Toeplitz operator. There appears to be a smoothing effect. 
\end{abstract}

\maketitle


\section{Introduction}

Let $H^2$ denote the Hardy space of the open unit disk $\mathbb{D} = \{|z| < 1\}$ and $L^2 = L^2(d\theta/2 \pi)$ denote the classical Lebesgue space of the unit circle $\T= \{|z| = 1\}$ with norm $\|\cdot\|$. $H^2$ is regarded as a closed subspace of $L^2$ in the usual way via non-tangential boundary values. For an inner function $I$, we let 
$K_{I} = H^2 \ominus I H^2$ be the well-known model space \cite{Niktr}. 

The boundary behavior of functions in $K_I$ have been well studied. For example, every function in $K_{I}$ has a meromorphic pseudo-continuation to the extended exterior disk \cite{CR, DSS, SR02}: For every $f \in K_I$, there is a meromorphic function $F$ on the extended exterior disk  whose non-tangential boundary values match those of $f$ almost everywhere. 
As another example \cite{Mo}, every $f \in K_{I}$ has an analytic continuation 
across $\T \setminus \sigma(I)$, where 
$$\sigma(I) = \left\{|z| \leq 1: \varliminf_{\lambda \to z} |I(\lambda)| = 0\right\}$$ 
is the spectrum of $I$. 
If $I = B_{\Lambda} s_{\mu}$, where $B_{\Lambda}$ is the Blaschke factor with zeros $\Lambda = \{\lambda_n\}_{n \geq 1} \subset \D$ (repeated according to multiplicity) and $s_{\mu}$ is the singular inner factor with associated singular measure $\mu$ on $\T$, then 
$$\sigma(I) = \Lambda^{-} \cup \mbox{suppt}(\mu).$$
Note that every function in $K_{I}$ has a pseudo-continuation across $\T$ although, if the Blaschke product has zeros which accumulate everywhere on $\T$ or if the support of $\mu$ is all of $\T$, for example, functions in $K_I$ might not have an analytic continuation across any subarc of $\T$. 

Our starting point for this paper is a result of Ahern and Clark \cite{AC70} which examines the non-tangential boundary behavior of functions in $K_I$ even closer by considering what happens near $\sigma(I)$ where analytic continuation is not guaranteed. To state their result, we set a bit of notation: Let $P_I$ be the orthogonal projection of $L^2$ onto $K_I$ and $A_z: K_{I} \to K_{I},  A_z f = P_{I}(z f)$ be the compression of the shift (`multiplication by $z$' on $H^2$) to $K_I$. 

\begin{theorem}[\cite{AC70}]  \label{AC-paper}
For an inner function $I = B_{\Lambda} s_{\mu}$ and $\zeta \in \T$, the following are equivalent: 
\begin{enumerate}
\item Every $f \in K_{I}$ has a non-tangential limit at $\zeta$, i.e., 
$$f(\zeta) := \angle \lim_{\lambda \to \zeta} f(\lambda)$$ exists. 

\item For every $f \in K_I$, $f(\lambda)$ is bounded as $\lambda \to \zeta$ non-tangentially. 
\item $P_{I} 1 \in \Rng(\Id - \overline{\zeta} A_z)$.
\item $(\Id - \overline{\lambda} A_z)^{-1} P_{I} 1$ is norm bounded as $\lambda \to \zeta$ non-tangentially. 
\item $I$ has an angular derivative in the sense of Caratheodory at $\zeta$, i.e., 
$$\angle \lim_{\lambda \to \zeta} I(\lambda)  = \eta \in \T$$ and 
$$\angle \lim_{\lambda \to \zeta} I'(\lambda) \mbox{\; \; exists}.$$
\item The following two conditions hold:
\begin{equation} \label{AC1}
\sum_{n \geq 1} \frac{1 - |\lambda_{n}|^2}{|\zeta - \lambda_{n}|^2} < \infty
\end{equation}
\begin{equation} \label{AC2}
\int_{\T} \frac{d \mu(\xi)}{|\xi - \zeta|^2} < \infty.
\end{equation}
\end{enumerate}
\end{theorem}

This is only a partial statement of the Ahern-Clark result. They went on further to characterize the
existence of non-tangential boundary limits of the derivatives (up to a given order) of functions in $K_I$.

Note that simple examples show that one can have an inner function $I$ and a $\zeta \in \T$ such that every function in $K_I$ has a non-tangential limit at $\zeta$ without necessarily having an analytic continuation to a neighborhood of $\zeta$. 

If one (and hence all) of the equivalent conditions of the Ahern-Clark theorem is 
satisfied, then it makes sense to evaluate functions $f\in K_I$ at $\zeta$, and
the corresponding point evaluation functional can be represented by $k_{\zeta}^I$. That is to say that
$$f(\zeta)  = \langle f, k_{\zeta}^{I} \rangle \quad \forall f \in K_{I}.$$

In this paper, we study the boundary values of functions in $K_I$ even further - beyond pseudo-continuation, analytic continuation, or the above Ahern-Clark result - by replacing the function $P_I 1$ in conditions (3) and (4) in the Ahern-Clark theorem with $P_I h$ where $h \in H^2$.

Let us take a closer look at $(\Id - \overline{\lambda} A_z)^{-1} P_I 1$ from condition (4). Since $(A_z)^{n} g = P_I z^n g$ for any $g \in K_I$, we get, for every $f\in K_I$ and $\lambda\in \D$,
\bea\label{formally0}
 \langle f, (\Id - \overline{\lambda} A_z)^{-1} P_I 1 \rangle & = &\left\langle f, \sum_{n = 0}^{\infty} \overline{\lambda}^n (A_z)^n P_I 1\right\rangle
  = \left\langle f, \sum_{n = 0}^{\infty} \overline{\lambda}^n z^n P_I 1 \right\rangle
 = \left\langle f, \frac{1}{1 - \overline{\lambda} z} P_I 1\right\rangle\nn\\
 &=& \left\langle f, \frac{1}{1 - \overline{\lambda} z} (P_I 1 - 1) + 
    \frac{1}{1 - \overline{\lambda} z}1 \right\rangle
 = \left\langle f, \frac{1}{1 - \overline{\lambda} z} 1 \right\rangle\nn\\
 &=& f(\lambda).
\eea
Thus $(\Id - \overline{\lambda} A_z)^{-1} P_I 1$ is the reproducing kernel $k_{\lambda}^{I}$ for the model space $K_I$ and the Ahern-Clark theorem gives a condition as to when $k_{\lambda}^{I}$ converges weakly to the boundary reproducing kernel function $k_{\zeta}^I$ as $\lambda \to \zeta$ non-tangentially. 

When $P_I 1$ is replaced by $P_I h$, where $h \in H^2$, an analogous calculation to the one in \eqref{formally0} 
gives us, at least formally, 
\begin{equation} \label{formally}
\langle f, (\Id - \overline{\lambda} A_z)^{-1} P_I h \rangle = (A_{\overline{h}} f)(\lambda),
\end{equation}
where $ A_{\overline{h}} f = P_{I}(\overline{h} f)$ is the truncated Toeplitz operator on $K_I$ - which we assume to be bounded. Note that $A_{\overline{h}}$ is initially densely defined on bounded functions on $K_I$ and, for certain $h$, can be extended to be bounded on $K_I$. Certainly if $h \in H^{\infty}$, the bounded analytic functions on $\D$, then $A_{\overline{h}}$ is bounded on $K_I$. However, there are unbounded $h \in H^2$ which yield bounded $A_{\overline{h}}$. We will discuss these details further in the next section.
Truncated Toeplitz operators have been studied quite a lot recently and we refer the reader to the seminal paper by Sarason which started it all \cite{Sa07}.

By examining the weak convergence of the kernel functions 
\begin{equation} \label{kfAh}
k_{\lambda}^{h} := (\Id - \overline{\lambda} A_z)^{-1} P_I h
\end{equation}
 as $\lambda \to \zeta$ (non-tangentially), we will determine the boundary behavior of functions in $\Rng A_{\overline{h}}$, the range of the truncated Toeplitz operator $A_{\overline{h}}$. 
Since $\Rng A_{\overline{h}} \subset K_I$, functions in this range will have finite non-tangential limits at at all points $\zeta \in \T$ where conditions \eqref{AC1} and \eqref{AC2} are satisfied. Certain choices of $h$  can force other points $\zeta \in \mathbb{T}$ to be points of finite non-tangential limits. In the Section \ref{S5} of this paper, we will make a few remarks about the boundary behavior of the functions 
$$f_{h}(\lambda) := \langle f, (\Id - \overline{\lambda} A_z)^{-1} P_I h \rangle$$
(which is the left-hand side of \eqref{formally}), where the truncated Toeplitz operator $A_{\overline{h}}$ is not necessarily bounded and $f \in K_I$ is not necessarily in the domain of $A_{\overline{h}}$.

To state our main theorem, we introduce 
some notation. For $\lambda \in \mathbb{D}$, let 
$$b_{\lambda}(z) = \frac{z - \lambda}{1 - \overline{\lambda} z}$$ be the single Blaschke factor 
with zero at $\lambda$. For a Blaschke product $B_{\Lambda}=\prod_{\lambda\in\Lambda}
(|\lambda|/\lambda)b_{\lambda}$ with zeros $\Lambda = \{\lambda_{n}\}_{n \geq 1}$, repeated
 accordingly to multiplicity, let the Takenaka-Malquist-Walsh functions be defined by
$$\gamma_{n}(z) = \frac{\sqrt{1 - |\lambda_n|^2}}{1 - \overline{\lambda_{n}} z} \prod_{k = 1}^{n - 1} b_{\lambda_{k}}(z).$$
It is well known \cite[p.~117]{Niktr} \cite{Tak} that $\{\gamma_{n}: n \in \mathbb{N}\}$ is an orthonormal basis for $K_{B_{\Lambda}}$. In fact, this basis was used in the proof of the Ahern-Clark theorem mentioned earlier. 
With our notation set, our main result reads as follows.

\begin{theorem} \label{MT-one}
When $I$ is a Blaschke product
with zeros $\Lambda = \{\lambda_{n}\}_{n \geq 1}$ and $h \in H^2$ so that $A_{\overline{h}}$ is bounded on $K_{I}$, every function in $\Rng A_{\overline{h}}$ has a finite non-tangential limit at $\zeta \in \T$ if and only if 
\begin{equation} \label{cond-xx}
\sum_{n \geq 1} |(A_{\overline{h}} \gamma_{n})(\zeta)|^2 < \infty.
\end{equation}
\end{theorem}

The alert reader might question whether or not $(A_{\overline{h}} \gamma_n)(\zeta)$ in \eqref{cond-xx} actually exists. It is after all the non-tangential boundary value of a function from $K_I$. However, as we will see in the proof of this theorem, $A_{\overline{h}} \gamma_n$ will turn out to be a rational function whose poles lie outside of $\D^{-}$ and so $A_{\overline{h}} \gamma_n$ can be evaluated at $\zeta$ without any difficulty. 
Also observe that when $h = 1$, 
$$\sum_{n \geq 1} |(A_{\overline{h}} \gamma_{n})(\zeta)|^2 = \sum_{n \geq 1} |\gamma_n(\zeta)|^2 = \sum_{n \geq 1} \frac{1 - |\lambda_{n}|^2}{|\zeta - \lambda_{n}|^2},$$ giving us condition \eqref{AC1} in the Ahern-Clark theorem. 

The proof of Theorem \ref{MT-one} will show that when condition \eqref{cond-xx} is satisfied then, as $\lambda \to \zeta$ non-tangentially,  the kernel functions $k^{h}_{\lambda}$ from \eqref{kfAh} converge weakly to some function $k^{h}_{\zeta} \in K_{I}$. This function turns out to be sort of a reproducing kernel for $\Rng A_{\overline{h}}$ at $\zeta$ in that 
$$(A_{\overline{h}} f)(\zeta) = \langle f, k^{h}_{\zeta} \rangle \quad \forall f \in K_{I}.$$ We will see from the proof of Theorem \ref{MT-one} that 
$$\|k^{h}_{\zeta}\|^2 = \sum_{n \geq 1} |(A_{\overline{h}} \gamma_{n})(\zeta)|^2.$$

In Section \ref{S3} we will compute an explicit formula for $A_{\overline{h}}\gamma_n(\zeta)$
which turns out to be quite cumbersome in the general case.
Still, we are able to give some examples in Section \ref{examples} of when the condition in \eqref{cond-xx} holds. We mention that when $I$ is an interpolating Blaschke product \cite[Ch.~VII]{Garnett}, the condition in \eqref{cond-xx} 
becomes much simpler. 

\begin{theorem} \label{MT-two}
If $I$ is an interpolating Blaschke product with zeros $\Lambda = \{\lambda_{n}\}_{n \geq 1}$ and $h \in H^2$ such that $A_{\overline{h}}$ is bounded on $K_I$, then every function in $\Rng A_{\overline{h}}$ has a finite non-tangential limit at $\zeta \in \T$ if and only if 
\begin{equation}  \label{eq1p29}
\sum_{n \geq 1} (1 - |\lambda_{n}|^2) \left|\frac{h(\lambda_n)}{\zeta - \lambda_n}\right|^2 < \infty.
\end{equation}
\end{theorem}

We will discuss an example in Section
\ref{examples} which will show that this condition does, in ge\-neral, not apply
to non-interpolating Blaschke products.
In fact, it
already fails when we take a Blaschke product associated with a non-separated 
union of two interpolating sequences. 
Although we do not develop this further here, the corresponding example
will show how one can obtain a condition for finite unions of interpolating Blaschke
products.\\

Non-tangential boundary values of functions in spaces related to backward shift invariant subspaces
have been studied recently. We would like to mention in particular the results
by Fricain and Mashreghi dealing with de Branges-Rovnyak spaces $\mathcal{H}(b)$
 \cite{Fricain-Mashreghi, Fricain-Mashreghi2} which are one way of generalizing the backward
shift invariant subspaces.
See Sarason's book \cite{Sa} for relations between the spaces
$\mathcal{M}(\overline{a}):=T_{\overline{a}}H^2$ and $\mathcal{H}(b)$ when $b$ is non extreme (this
guarantees that there is $a\in \operatorname{Ball}(\Hi)$ such that $|a|^2+|b|^2=1$). Our situation is somewhat different
since we consider Toeplitz operators not on the whole $H^2$ but only on the
model space $K_I$.

Finally, the first mentioned author has considered analytic
continuation questions in weighted backward shift invariant subspaces which
appear naturally in the context of kernels of Toeplitz operators \cite{AH10}.
We refer the reader to the survey \cite{FH} for more information.

The reader has probably noticed that we only discuss inner functions $I$ which are Blaschke products, i.e., $I$ has no singular inner factor. We will make some comments at the end of the paper as to the difficulties which arise in the the presence of a singular inner factor. 

A final word concerning numbering in this paper: in each section, we have numbered theorems, propositions,
lemmas, corollaries {\bf and} equations consecutively.

\section{Preliminaries}\label{S2}

For an inner function $I$, let $K_I = H^2 \ominus I H^2$ be the model space corresponding to $I$. Since $H^2$ is a reproducing kernel Hilbert space with kernel 
$$k_{\lambda}(z) = \frac{1}{1 - \overline{\lambda} z},$$ then so is $K_I$ with reproducing kernel 
$$k_{\lambda}^{I}(z) = (P_I k_{\lambda})(z) = \frac{1 - \overline{I(\lambda)} I(z)}{1 - \overline{\lambda} z},$$ where $P_I$ is the orthogonal projection of $L^2$ onto $K_I$. Note that these kernels are bounded functions and finite linear combinations of them form a dense subset of $K_I$. This enables us, for $\vp \in L^2$, to define the operator 
$A_{\vp}$ densely on $K_I$ by $A_{\vp} f = P_{I}(\vp f)$. These operators, called \emph{truncated Toeplitz operators}, have many interesting properties \cite{Sa07} which we won't get into here. We do, however, mention a few of them which will be important for our purposes. 

First we note that the symbols which represent truncated Toeplitz operators are not unique. In fact \cite[Thm.~3.1]{Sa07}
\begin{equation} \label{id-zero}
A_{\vp_1} = A_{\vp_2} \Leftrightarrow \vp_1 - \vp_2 \in I H^2 + \overline{I H^2}.
\end{equation}
Secondly, when $\vp$ is a bounded function then certainly the truncated Toeplitz operator $A_{\vp}$ extends to be a bounded operator on $K_I$ with $\|A_{\vp}\| \leq \|\vp\|_{\infty}$. However, there are bounded truncated Toeplitz operators (i.e., ones which extend to be bounded on $K_I$) which do not have a bounded symbol \cite{Bar}. 

In this paper, we focus our attention on the co-analytic truncated Toeplitz operator $A_{\overline{h}}$, where $h \in H^2$. As mentioned earlier, when $h \in H^{\infty}$, the bounded analytic functions on $\D$, then $A_{\overline{h}}$ is bounded on $K_I$. Although by using \eqref{id-zero} every bounded  $A_{\overline{h}}$ has an unbounded symbol,
a well-known result of Sarason \cite{Sa67} says that if a co-analytic truncated Toeplitz operator is bounded, then it can be represented by a \emph{bounded} co-analytic symbol. 

The central step in the Ahern-Clark approach is to express the reproducing kernel $k^{I}_{\lambda}$
in terms of the resolvent of a certain operator in $\lambda$ applied to a fixed function:
\[
 k^I_{\lambda}=(\Id-\overline{\lambda} A_z)^{-1} P_{I} 1.
\]
 In this situation, the following lemma
allows to deduce the existence
of the boundary limits at a point $\zeta\in \T$ from the 
fact that $(\Id-\overline{\zeta} A_z)$ is injective and $P_I 1$ is in the range of this
operator.

\begin{lemma}[\cite{AC70}]\label{AClemma}
Let $\xi \in\T$ and $L$ be a contraction on a Hilbert space $\mathcal{H}$ such that 
$(\Id-\xi L)$ is injective.  Furthermore, let $\{\lambda_n\}_{n \geq 1}$  be sequence
of points in $\D$ tending non-tangentially to $\xi$ as $n\to\infty$. Then, for a fixed
 $y\in \mathcal{H}$, the sequence
\[
 w_n=(\Id-\lambda_n L)^{-1}y
\]
is uniformly bounded if and only if $y$ belongs to the range of
$(\Id-\xi L)$,
in which case, $w_n$ tends weakly to $w_0=(\Id-\xi L)^{-1}y$.
\end{lemma}

\begin{remark}
Below we will apply this lemma to the operator $A_z$ on $K_{I}$. Clearly $A_z$ is a contraction on $K_I$. To show that $(\Id - \xi A_{z})$ is injective, observe, for  $f \in K_I$, that
$$(\Id - \xi A_{z}) f = 0 \Leftrightarrow P_I ((1 - \xi z) f) = 0 \Leftrightarrow (1 - \xi z) f \in I H^2.$$
But since $z \mapsto (1 - \xi z)$ is an outer function, then $I$ divides the inner part of $f$ from which we get $f \in I H^2$ and so, since $f \in K_I = H^{2} \ominus I H^2$, $f \equiv 0$. 
\end{remark}

As mentioned in \eqref{formally}, for $h \in H^2$, the function 
$$k^{h}_{\lambda} = (\Id - \overline{\lambda} A_{z})^{-1} P_I h$$ serves as a reproducing kernel for 
$\Rng A_{\overline{h}}$ in the sense that
\begin{equation} \label{AhfI}
(A_{\overline{h}} f)(\lambda) = \langle f, k^{h}_{\lambda} \rangle, \quad f \in K_I.
\end{equation}
From this and the identity $$(A_{\overline{h}}f)(\lambda)=\langle P_I(\overline{h}f),k_{\lambda}\rangle
=\langle f,hk_{\lambda}^I\rangle=\langle f,P_I(hk_{\lambda}^I)\rangle, \quad \forall f \in K_I,$$
we also deduce that
\bea\label{klh}
 k_{\lambda}^h=P_I(hk_{\lambda}^I).
\eea

The next proposition, similar to Theorem \ref{AC-paper}, begins to get at the boundary behavior of functions in $\Rng A_{\overline{h}}$. The proof is pretty much the same but we include it anyway for the sake of completeness. 

\begin{proposition} \label{AC-Th}
For an inner function $I$, a point $\zeta \in \T$, and a function $h \in H^2$ so that $A_{\overline{h}}$ is bounded on $K_I$, the following are equivalent:
\begin{enumerate}
\item Every function in $\Rng A_{\overline{h}}$ has a finite non-tangential limit at $\zeta$. 
\item  $P_I h\in\Rng (\Id - \overline{\zeta} A_z)$.
\item $k^{h}_{\lambda}$ is norm bounded as $\lambda \to \zeta$ non-tangentially. 
\end{enumerate}
\end{proposition}

\begin{proof}
By \eqref{AhfI}, along with the uniform boundedness principle, we have (1) implies (3). Statement (3) is equivalent to (2) by Lemma \ref{AClemma}. Statement (3) implies (1) follows from Lemma \ref{AClemma} and \eqref{AhfI}.
\end{proof}

\begin{corollary} \label{C-int-prob}
The following statements are equivalent:
\begin{enumerate}
\item Every function in $\Rng A_{\overline{h}}$ has a finite non-tangential limit at $\zeta$. 
\item There exists $u \in H^2$ and $k \in K_I$ which solve the following interpolation problem
\begin{equation} \label{Pre-int}
P_I h = (1 - \overline{\zeta} z) k + I u.
\end{equation}
\end{enumerate}
\end{corollary}

\begin{proof}
Assuming statement (1) holds, we can use Proposition \ref{AC-Th} along with Lemma \ref{AClemma} to say that $k^{h}_{\lambda}$ converges weakly to some $k_{\zeta}^{h} \in K_I$ as $\lambda \to \zeta$ non-tangentially and moreover,
$$k_{\zeta}^{h} = (\Id - \overline{\zeta} A_{z})^{-1} P_I h.$$
Using the  the general observation
$P_I(zv)-zv=(P_I-\Id)(zv)\in\Ker P_I=IH^2$ we see that 
\begin{align*}
P_I h & = (\Id - \overline{\zeta} A_{z}) k^{h}_{\zeta}\\
& = k^{h}_{\zeta} - \overline{\zeta} A_{z} k^{h}_{\zeta}\\
& = k^{h}_{\zeta} - \overline{\zeta} z k^{h}_{\zeta} + I u, \quad u \in H^2\\
& = (1 - \overline{\zeta} z) k^{h}_{\zeta} + I u.
\end{align*}
This shows that (1) implies (2). To show (2) implies (1), simply reverse the argument. 
\end{proof}

The above proof also says the following. 

\begin{corollary} \label{C-eval}
If $A_{\overline{h}} f$ has a finite non-tangential limit at $\zeta$ for every $f \in K_I$ then 
$$(A_{\overline{h}} f)(\zeta) = \langle f, k^{h}_{\zeta} \rangle.$$
\end{corollary}

\begin{proof}
In this situation, using \eqref{klh}, we will have,
for every $f\in K_I$,
\beqa\label{eq2p26}
 \langle f, k_{\zeta}^h\rangle
 &=& \angle \lim_{\lambda \to \zeta}\langle f,k_{\lambda}^h\rangle
 = \angle \lim_{\lambda \to \zeta}\langle f,P_I(hk_{\lambda})\rangle
 = \angle \lim_{\lambda \to \zeta}\langle f,hk_{\lambda}\rangle
 = \angle \lim_{\lambda \to \zeta}\langle \overline{h}f,k_{\lambda}\rangle\nn\\
 &=&(A_{\overline{h}}f)(\zeta). 
\eeqa
\end{proof}

\section{The main results}\label{S3}

\begin{remark}
Until we say otherwise, we will assume that $h \in H^2$ is chosen so that $A_{\overline{h}}$ is bounded on $K_I$. Furthermore,  by \eqref{id-zero},  $A_{\overline{h}} = A_{\overline{P_I h}}$ and so we will also assume that $h \in K_I$. 
\end{remark}

We will proceed as in \cite{AC70}. For a Blaschke product $I$ with zero set
$\Lambda=\{\lambda_n\}_{n\ge 1}$, we have already introduced the functions
\[
 \gamma_n(z)=\frac{\sqrt{1-|\lambda_n|^2}}{1-\overline{\lambda}z}
 \underbrace{\prod_{k=1}^{n-1}b_{\lambda_k}(z)}_{=:B_{n-1}(z)}
\]
which form an orthonormal basis for $K_I$.

It turns out that the central point in the result is the behavior of $A_{\overline{h}} \gamma_n$ at a boundary point.  This is what we will determine now.
Before proceeding though, we should justify that the
expression $(A_{\overline{h}}\gamma_n)(\zeta)$ is always defined.
First observe that $\gamma_n$ belongs to $K_{B_{n}}$, a finite dimensional subspace of rational functions whose poles lie outside $\D^{-}$. Moreover, $A_{\overline{h}} \gamma_n \in K_{B_{n}}$. 
This is because $A_{\overline{h}}$ acts on $K_{B_n}$
as the restriction of the co-analytic Toeplitz operator $T_{\overline{h}}$,
and $T_{\overline{h}}  K_{B_{n}} \subset K_{B_{n}}$. 
 Consequently, we can evaluate $A_{\overline{h}} \gamma_n$
at $\zeta\in\T$ without any difficulty.

\begin{proposition}\label{prop2.2}
let $\Lambda$ be a Blaschke sequence and $h\in H^2$. Then, writing $$\prod_{l=1}^n(z-\lambda_l)=\prod_{l=1}^r (z-\mu_l)^{k_l},$$
we have, for any $\zeta\in\T$, 
\bea\label{eq1p6a}
 (A_{\overline{h}}\gamma_n)(\zeta)
 =\sqrt{1-|\lambda_n|^2} \sum_{l=1}^r 
 \frac{1}{(k_l-1)!} \overline{\frac{d^{k_l-1}}{d\mu_l^{k_l-1}}
 \left[\frac{h(\mu_l)\prod_{m=1}^{n-1}
  (1-\overline{\lambda_m}\mu_l)}{(1-\overline{\zeta}{\mu_l})
 \prod_{j=1,j\neq l}^r (\mu_l-\mu_j)^{k_j}}\right]}.
\eea
\end{proposition}

\begin{proof}
Since $A_{\overline{h}} = T_{\overline{h}}|K_I$ and $T_{\overline{h}} K_{B_n} \subset K_{B_n}$ we get, for $\lambda \in \D$, 
$$(A_{\overline{h}} \gamma_{n})(\lambda) = (T_{\overline{h}} \gamma_{n})(\lambda) = (P_{+} \overline{h} \gamma_n)(\lambda).$$
This last quantity is now equal to 
\[
\langle \overline{h}\gamma_n,k_{\lambda}\rangle
 =\sqrt{1-|\lambda_n|^2}\langle k_{\lambda_n}B_{n-1},hk_{\lambda}\rangle.
\]
We thus have to compute 
\[
 \langle k_{\lambda_n}B_{n-1},hk_{\lambda}\rangle
 =\int_{\T}^{}\frac{1}{1-\overline{\lambda}_nz}\prod_{l=1}^{n-1}
 \frac{z-\lambda_l}{1-\overline{\lambda_l}z} \overline{h(z)}\frac{1}{1-\lambda\overline{z}}
 dm(z).
\]
Passing to the conjugate expression and then replacing the measure $dm = d \theta/(2 \pi)$ by $dz/(2 \pi i z)$ we get
\bea\label{eq1p6b}
 \overline{\langle k_{\lambda_n}B_{n-1},hk_{\lambda}\rangle}
 &=&\frac{1}{2\pi i}\int_{\T}^{}\frac{1}{1-\overline{z}\lambda_n}\prod_{l=1}^{n-1}
 \frac{1-\overline{\lambda_l}z}{z-\lambda_l}{h(z)}\frac{1}{1-\overline{\lambda}{z}}
 \frac{dz}{z}\nn\\
 &=&\frac{1}{2\pi i}\int_{\T}^{}\frac{1}{z-\lambda_n}\prod_{l=1}^{n-1}
 \frac{1-\overline{\lambda_l}z}{z-\lambda_l}{h(z)}\frac{1}{1-\overline{\lambda}{z}}
 dz\nn\\
 &=&\frac{1}{2\pi i}\int_{\T}^{}\prod_{l=1}^n\frac{1}{z-\lambda_l}
 \left[\frac{h(z)\prod_{j=1}^{n-1}
  (1-\overline{\lambda_j}z)}{1-\overline{\lambda}{z}}\right]
 dz.
\eea
Now let $\prod_{l=1}^n(z-\lambda_l)=\prod_{l=1}^r (z-\mu_l)^{k_l}$ where
$\mu_l$ are the \emph{different} zeros of $B_n$ and $k_l$ are their corresponding multiplicities.
Then from the residue theorem we obtain:
\beqa
 \overline{\langle k_{\lambda_n}B_{n-1},hk_{\lambda}\rangle}
  =\sum_{l=1}^r 
 \frac{1}{(k_l-1)!} \frac{d^{k_l-1}}{d\mu_l^{k_l-1}}
 \left[\frac{h(\mu_l)\prod_{m=1}^{n-1}
  (1-\overline{\lambda_m}\mu_l)}{(1-\overline{\lambda}{\mu_l})\prod_{j=1,j\neq l}^r  (\mu_l-\mu_j)^{k_j}}\right].
\eeqa
This expression is perfectly well behaved for $\lambda \lra\zeta$, so
that by conjugating back and multiplying by the normalization constant $\sqrt{1 - |\lambda_n|^2}$, we obtain
the desired result. 
\end{proof}

In the situation of simple zeros we get a much nicer formula that we
will use in the example at the end of this paper.

\begin{corollary}\label{prop6.3}
Let $\Lambda$ be a Blaschke sequence with simple zeros. Then
we have, for each 
$\zeta\in\T$,
\[
 (A_{\overline{h}}\gamma_n)(\zeta)=\sqrt{1-|\lambda_n|^2}
  \sum_{l=1}^n \frac{\overline{h(\lambda_l)}}   
 {1-\overline{\lambda_l}{\zeta}}\frac{1}{\overline{(B_n)_{\lambda_l}(\lambda_l)}}
 \frac{1-|\lambda_l|^2}{1-\overline{\lambda_l}\lambda_n}
\]
where
\[
 (B_n)_{\lambda_l}=\prod_{k=1,k\neq l}^nb_{\lambda_k}.
\]
\end{corollary}

The interesting observation here is that the expression $|(B_n)_{\lambda_l}(\lambda_l)|$
measures, in a sense, the deviation of $\Lambda$ from an interpolating sequence.
This will be very useful in our Example \ref{example4.3}.

\begin{proof}
Starting from the computation \eqref{eq1p6b} the residue theorem now gives:
\beqa
 \overline{\langle k_{\lambda_n}B_{n-1},hk_{\lambda}\rangle}
=\sum_{l=1}^n\frac{h(\lambda_l)\prod_{j=1}^{n-1}
  (1-\overline{\lambda_j}\lambda_l)}{1-\overline{\lambda}{\lambda_l}}\prod_{j=1,j\neq l}^n
 \frac{1}{\lambda_l-\lambda_j}.
\eeqa
We split the above sum in two pieces $l\le n-1$ and $l=n$ and do some
regrouping to get
\beqa
 \lefteqn{\sum_{l=1}^n\frac{h(\lambda_l)\prod_{j=1}^{n-1}
  (1-\overline{\lambda_j}\lambda_l)}{1-\overline{\lambda}{\lambda_l}}\prod_{j=1,j\neq l}^n
 \frac{1}{\lambda_l-\lambda_j}}\\
 &&=\sum_{l=1}^{n-1}\frac{h(\lambda_l)\prod_{j=1}^{n-1}
  (1-\overline{\lambda_j}\lambda_l)}{1-\overline{\lambda}{\lambda_l}}\prod_{j=1,j\neq l}^n
 \frac{1}{\lambda_l-\lambda_j}+\frac{h(\lambda_n)\prod_{j=1}^{n-1}
  (1-\overline{\lambda_j}\lambda_n)}{1-\overline{\lambda}{\lambda_n}}\prod_{j=1}^{n-1}
 \frac{1}{\lambda_n-\lambda_j}.
\eeqa
Now
\beqa
 \prod_{j=1}^{n-1}
  (1-\overline{\lambda_j}\lambda_l)\prod_{j=1,j\neq l}^n
 \frac{1}{\lambda_l-\lambda_j}
 &=&\frac{1-|\lambda_l|^2}{\lambda_l-\lambda_n}
 \prod_{j=1,j\neq l}^{n-1}\frac{1-\overline{\lambda_j}\lambda_l}{\lambda_l-\lambda_j}
 =\frac{1-|\lambda_l|^2}{\lambda_l-\lambda_n}
 \left(\prod_{j=1,j\neq l}^{n-1}\frac{1-\overline{\lambda_j}\lambda_l}{\lambda_l-\lambda_j}\right)
 \frac{1-\overline{\lambda_n}\lambda_l}{1-\overline{\lambda_n}\lambda_l}\\
 &=&\frac{1-|\lambda_l|^2}{1-\overline{\lambda_n}\lambda_l}
 \frac{1}{(B_n)_{\lambda_l}(\lambda_l)}.
\eeqa
Also,
\[
 \prod_{j=1}^{n-1}
  (1-\overline{\lambda_j}\lambda_n)\prod_{j=1}^{n-1}
 \frac{1}{\lambda_n-\lambda_j}
 =\frac{1}{(B_n)_{\lambda_n}(\lambda_n)}
 =\frac{1-|\lambda_n|^2}{1-\overline{\lambda_n}\lambda_n}
 \frac{1}{(B_n)_{\lambda_n}(\lambda_n)}.
\]
Hence
\beqa
 \lefteqn{\sum_{l=1}^n\frac{h(\lambda_l)\prod_{j=1}^{n-1}
  (1-\overline{\lambda_j}\lambda_l)}{1-\overline{\lambda}{\lambda_l}}\prod_{j=1,j\neq l}^n
 \frac{1}{\lambda_l-\lambda_j}}\\
 &&=\sum_{l=1}^{n-1}\frac{h(\lambda_l)} {(1-\overline{\lambda_j}\lambda_l)}
 \frac{1-|\lambda_l|^2}{1-\overline{\lambda_n}\lambda_l}
 \frac{1}{(B_n)_{\lambda_l}(\lambda_l)}
 +\frac{h(\lambda_n)}{ (1-\overline{\lambda_j}\lambda_n)}
\frac{1-|\lambda_n|^2}{1-\overline{\lambda_n}\lambda_n}
 \frac{1}{(B_n)_{\lambda_n}(\lambda_n)}
\eeqa
which concludes the proof.
\end{proof}

\begin{remark}
It is worth reminding the reader again that we are assuming $h \in K_I$ and $A_{\overline{h}}$ is bounded on $K_I$. 
\end{remark}

\begin{proof}[Proof of Theorem \ref{MT-one}]
By Corollary \ref{C-int-prob} the existence of finite non-tangential boundary limits of all 
functions in $\Rng A_{\overline{h}}$ is equivalent to the interpolation problem
of finding $k^h_{\zeta} \in K_I$ such that 
\bea\label{eq2p27}
 (1-\overline{\zeta}z)k^h_{\zeta}-h\in IH^2,
\eea
where $I$ is now a Blaschke product.

Let us use some ideas from \cite{AC70}. If there is a function $k^h_{\zeta}\in K_I$ 
satisfying \eqref{eq2p27}
then there are complex coefficients $c_n$ such that
\begin{equation} \label{qqq-1}
 k^h_{\zeta}=\sum_{n\ge 1} c_n\gamma_n
\end{equation}
with $\sum_{n\ge 1} |c_n|^2<\infty$.
In particular,
\[
 \overline{c}_n= \langle \gamma_n,k^h_{\zeta}\rangle .
\]
But since $\gamma_n\in K_I$ we can use Corollary \ref{C-eval} to get
\[
 \langle \gamma_n,k^h_{\zeta}\rangle=(A_{\overline{h}}\gamma_n)(\zeta)
\]
which proves the necessity.

Let us now prove the sufficiency. 

Assuming $\sum_{n\ge 1}|(A_{\overline{h}}\gamma_n)(\zeta)|^2<\infty$, 
we can define the function
\bea\label{deffctu}
 u=\sum_{n\ge 1} \overline{(A_{\overline{h}}\gamma_n)(\zeta)} \gamma_n
\eea
in $K_I$. In order to verify the interpolating condition in \eqref{eq2p27}, it is sufficient to
check that
\[
 u-\frac{h}{1-z\overline{\zeta}}
\]
vanishes to the right order, meaning that at each point $\lambda\in \Lambda$ these
differences vanish with order corresponding to the multiplicity of $\lambda$. 
The reader might observe that these differences are not necessarily in $H^2$. However,
it is clear that $h(z)/(1-\overline{\zeta}z)$ is controlled by $1/(1-\overline{\zeta}z)^{3/2}$
so that we can write the interpolation condition as 
\[ 
 u-\frac{h}{1-z\overline{\zeta}}\in I H^p
\]
for $p<2/3$, but we will not really use this formulation.

The proof of the interpolating condition will be very technical in the general
case. However, if the zeros are \emph{simple}, which we assume to be the case for the moment,  
then the formula for $A_{\overline{h}}\gamma_n(\zeta)$
in Corollary \ref{prop6.3} simplifies the argument considerably.
In this situation, we have
\[
 (A_{\overline{h}}\gamma_n)(\zeta)=\sqrt{1-|\lambda_n|^2}
  \sum_{l=1}^n \frac{\overline{h(\lambda_l)}}   
 {1-\overline{\lambda_l}{\zeta}}\frac{1}{\overline{(B_n)_{\lambda_l}(\lambda_l)}}
 \frac{1-|\lambda_l|^2}{1-\overline{\lambda_l}\lambda_n}.
\]
Hence using Fubini's theorem we get, for each $N\in\N$,
\bea\label{eq1p9}
  u(\lambda_N)&=&\sum_{n=1}^N \overline{(A_{\overline{h}}\gamma_n)(\zeta)} 
  \gamma_n(\lambda_N)\nn\\
 &=&\sum_{n=1}^N \sqrt{1-|\lambda_n|^2}
  \sum_{l=1}^n \frac{{h(\lambda_l)}}   
 {1-{\lambda_l}\overline{\zeta}}\frac{1}{{(B_n)_{\lambda_l}(\lambda_l)}}
 \frac{1-|\lambda_l|^2}{1-\overline{\lambda_n}\lambda_l}
 \frac{\sqrt{1-|\lambda_n|^2}}{1-\overline{\lambda_n}\lambda_N}
 B_{n-1}(\lambda_N)\nn\\
 &=&\sum_{l=1}^N \frac{{h(\lambda_l)}}   
  {1-{\lambda_l}\overline{\zeta}}
 \underbrace{\sum_{n=l}^N 
   \frac{1-|\lambda_l|^2}{1-\overline{\lambda_n}\lambda_l}
 \frac{{1-|\lambda_n|^2}}{1-\overline{\lambda_n}\lambda_N}
 \frac{B_{n-1}(\lambda_N)}{{(B_n)_{\lambda_l}(\lambda_l)}}}_{\dst =:\alpha_{l,N}}.
\eea
So, in order to show the interpolation condition $u(\lambda_N)=h(\lambda_N)/(1-\overline{\zeta}
\lambda_N)$ it suffices to show that 
\[
 \alpha_{l,N}=\left\{\begin{array}{ll}
 1 & \mbox{if }l=N\\
 0 &\mbox{if }l<N.
 \end{array}
 \right.
\]
Clearly, if $l=N$ then $\alpha_{N,N}=1$ (observe in particular that 
$l=n=N$ and $(B_N)_{\lambda_N}=B_{N-1}$).

Now let $k_{r\xi}(z)=1/(1-r\overline{\xi}z)$ be the reproducing kernel for $H^2$ at $r\xi$ for
any $\xi\in\T$. Let $P_{B_N}$ be the orthogonal projection onto $K_{B_N}$ which can be written explicitly using the Takenaka-Malmquist-Walsh basis so that 
\[
 v_r:=P_{B_N}k_{r\xi} =\sum_{n=1}^n\langle k_{r\xi},\gamma_n\rangle\gamma_n
 =\sum_{n=1}^N\overline{\gamma_n(r\xi)}\gamma_n.
\]
Since $v_r-k_{r\xi}\in \ker P_{B_N}=B_NH^2$ we get $v_r(\lambda_n)=k_{r\xi}(\lambda_n)$
for $n=1,\ldots,N$. All functions involved are rational function with no poles on $\D^{-}$ so that we can pass
to the limit as $r\to 1^{-}$ so that
\[
 v(\lambda_n)=\lim_{r\to 1^{-}}v_r(\lambda_n)=\frac{1}{1-\overline{\xi}\lambda_n},
 \quad n=1,2,\ldots, N.
\]
Notice also that 
\[
 v=\sum_{n=1}^N\overline{\gamma_n(\xi)}\gamma_n
 =\sum_{n=1}^N\overline{(A_{\overline{1}}\gamma_n)(\xi)}\gamma_n.
\]
Replacing the function $h$ by 1 in \eqref{eq1p9}, we obtain
\beqa
  v(\lambda_N)
 =\sum_{l=1}^N \frac{{1}}   
  {1-{\lambda_l}\overline{\xi}}
 \underbrace{\sum_{n=l}^N 
   \frac{1-|\lambda_l|^2}{1-\overline{\lambda_n}\lambda_l}
 \frac{{1-|\lambda_n|^2}}{1-\overline{\lambda_n\lambda_N}}
 \frac{B_{n-1}(\lambda_N)}{{(B_n)_{\lambda_l}(\lambda_l)}}}_{\dst =\alpha_{l,N}},
\eeqa
and since $v(\lambda_N)=1/(1-\overline{\zeta}\lambda_N)$ and $\alpha_{N, N} = 1$, we get
\[
 \sum_{l=1}^{N-1} \frac{{1}}   
  {1-{\lambda_l}\overline{\xi}}
  \alpha_{l,N}=0
\]
for every $\xi$. The reproducing kernels for different $\xi$ are linearly independent,
so that the coefficients $\alpha_{l,N}$ must necessarily vanish for $l=1,2,\ldots, N-1$,
which finishes the proof for simple zeros.
\\

The reader might observe that the explicit form of $\alpha_{l,N}$ is not really
of importance (well, it is, of course...). The central point is that $\alpha_{N,N}=1$.
We will now generalize this argument to the case of arbitrary Blaschke products.
As to be expected, the proof is more technical.
\\

For the proof in the general situation,
let  $\mu=\lambda_{N+1}$ be any point of the sequence such that $\mu\neq \lambda_l$
for every $l\le N$. It is the first time we meet this zero. Suppose also that $\mu$ has 
multiplicity $k_0$. We have to show that for every $1\le k \le k_0$,
\[
  u^{(k-1)}(\mu)=\left(\frac{h}{1-\overline{\zeta}z}\right)^{(k-1)}(\mu).
\]
Let us compute the derivatives of $u$. 
Let $\prod_{n=1}^{N+k}(z-\lambda_n)=\prod_{l=1}^r(z-\mu_l)^{k_l}$
where $\mu_r=\mu$ and $k_r=k$
(and not $k_0$). Evaluating the $(k-1)$-st derivative of the function $u$, as defined 
in \eqref{deffctu}, at $\mu$ needs only to take into account the first $N+k$ terms
of the sum since for $n\ge N+k+1$, $\gamma_n$ has a zero of sufficiently high order
at $\mu$ that $\gamma_n^{(k-1)}(\mu)=0$. Thus from \eqref{eq1p6b} we get
\beqa
 \lefteqn{u^{(k-1)}(\mu)=\sum_{n=1}^{N+k}\overline{(A_{\overline{h}}\gamma_n)(\zeta)}
 \gamma_n^{(k-1)}(\mu)}\\
 &&=\sum_{n=1}^{N+k}(1-|\lambda_n|^2)\sum_{l=1}^r\frac{1}{(k_l-1)!}
  \frac{d^{k_l-1}}{d\mu_l^{k_l-1}}
 \left[\frac{h(\mu_l)\prod_{m=1}^{n-1}
  (1-\overline{\lambda_m}\mu_l)} {(1-\overline{\zeta}{\mu_l})
  \prod_{j=1,j\neq l}^r  (\mu_l-\mu_j)^{k_j} }\right]
 \left[k_{\lambda_n}B_{n-1}\right]^{(k-1)}(\mu)\\
 &&=\sum_{n=1}^{N+k}(1-|\lambda_n|^2)\sum_{l=1}^r\frac{1}{(k_l-1)!}
 \sum_{p=0}^{k_l-1}\binom{k_l-1}{p} 
 \frac{d^{p}}{d\mu_l^{p}}
 \left[\frac{h(\mu_l)}{1-\overline{\zeta}{\mu_l}  } \right]\times
 \\
 &&\hspace{2cm}\times \frac{d^{k_l-1-p}}{d\mu_l^{k_l-1-p}}
  \left[ \frac{\prod_{m=1}^{n-1}
  (1-\overline{\lambda_m}\mu_l)} {(1-\overline{\lambda}{\mu_l})
  \prod_{j=1,j\neq l}^r  (\mu_l-\mu_j)^{k_j} }\right]
 \left[k_{\lambda_n}B_{n-1}\right]^{(k-1)}(\mu).
\eeqa
We will now apply Fubini's theorem. In order to do this, we observe that the double
sum $\sum_{l=1}^r\sum_{p=0}^{k_l-1}$ runs exactly through the zeros
$\lambda_n$, $n=1,2,...,N+k$. Let us define a function in two variables by
\[
 \sigma(l,p)=(p+1)+\sum_{j=1}^{l-1}k_l
\]
which is a bijection of a disjoint union of sets $\tau_l=\{0,1\ldots,k_l-1\}$, $l=1,\ldots, r$
to the set $\{1,2,\ldots,N+k\}$.
Hence
\bea\label{eq1p9a}
 \lefteqn{u^{(k-1)}(\mu)=\sum_{n=1}^{N+k}\overline{(A_{\overline{h}}\gamma_n)(\zeta)}
 \gamma_n^{(k-1)}(\mu)}\nn\\
 &&=\sum_{l=1}^r\frac{1}{(k_l-1)!}
 \sum_{p=0}^{k_l-1}\binom{k_l-1}{p} 
 \frac{d^{p}}{d\mu_l^{p}}
 \left[\frac{h(\mu_l)}{1-\overline{\zeta}{\mu_l}  } \right]\times\nn\\
 && \hspace{2cm}\times\sum_{n=\sigma(l,p)}^{N+k}(1-|\lambda_n|^2)
  \frac{d^{k_l-1-p}}{d\mu_l^{k_l-1-p}}
  \left[ \frac{\prod_{m=1}^{n-1}
  (1-\overline{\lambda_m}\lambda_l)} {
  \prod_{j=1,j\neq l}^r  (\mu_l-\mu_j)^{k_j} }\right]
 \left[k_{\lambda_n}B_{n-1}\right]^{(k-1)}(\mu).
\eea

Let us investigate the term we are particularly interested in for the
interpolation problem. It corresponds to the very last term: $l=r$ and $p=k_r-1=k-1$.
In this situation, $n=\sigma(r,k-1)=N+k$.
We compute the last factor:
\[
 \left[k_{\lambda_n}B_{n-1}\right]^{(k-1)}(\mu)
 =\sum_{p=0}^{k-1}\binom{k-1}{p} k_{\lambda_{N+k}}^{(p)}B_{N+k-1}^{(k-1-p)}(\mu).
\]
Now $B_{N+k-1}=b_{\mu}^{k-1}\prod_{l=1}^{r-1}b_{\mu_l}^{k_l}$ so that all
derivatives up to order $k-2$ of this product evaluated at $\mu$ will vanish
and
\[
 B_{N+k-1}^{(k-1)}(\mu)=(b_{\mu}^{k-1})^{(k-1)}(\mu)\prod_{l=1}^{r-1}b_{\mu_l}^{k_l}(\mu).
\]
It is well known, and easy to verify (e.g.\ using once again the Leibniz rule), that
\[
 (b_{\mu}^{k-1})^{(k-1)}(\mu)=\frac{(k-1)!}{(1-|\mu|^2)^{k-1}}.
\]
Hence 
\[
 \left[k_{\lambda_n}B_{n-1}\right]^{(k-1)}(\mu)
 =k_{\mu}(\mu)\frac{(k-1)!}{(1-|\mu|^2)^{k-1}}\prod_{l=1}^{r-1}b_{\mu_l}^{k_l}(\mu)
 =\frac{(k-1)!}{(1-|\mu|^2)^{k}}\prod_{l=1}^{r-1}b_{\mu_l}^{k_l}(\mu).
\]
We are now in a position to compute the coefficient of the term
$\frac{\dst d^{k-1}}{\dst d\mu^{k-1}}\frac{\dst h(\mu)}{\dst 1-\overline{\zeta}\mu}$
(corresponding to $l=r$, $k_l=k$, $p=k-1$, and hence, as already seen, $n=\sigma(l,p)=N+k$,
$\lambda_{N+k}=\mu$). It is
given by
\bea\label{eq2p10}
\lefteqn{ \frac{1}{(k-1)!}\binom{k-1}{k-1}(1-|\mu|^2)
 \frac{d^{k-1-p}}{d\mu^{k-1-p}}
  \left[ \frac{\prod_{m=1}^{n-1}
  (1-\overline{\lambda_m}\mu)} {
  \prod_{j=1,j\neq l}^r  (\mu-\mu_j)^{k_j} }\right]
 \left[k_{\lambda_n}B_{n-1}\right]^{(k-1)}(\mu)}\nn\\
 &&= \frac{1}{(k-1)!}(1-|\mu|^2)
  \left[ \frac{\prod_{m=1}^{r-1}
  (1-\overline{\mu_m}\mu)^{k_m}(1-|\mu|^2)^{k-1}} {
  \prod_{j=1}^{r-1}  (\mu-\mu_j)^{k_j} }\right]
 \frac{(k-1)!}{(1-|\mu|^2)^{k}}\prod_{l=1}^{r-1}b_{\mu_l}^{k_l}(\mu)\nn\\
 &&=1.
\eea
Hence we are led to show that the remaining sum adds up to 0. For this,
it is sufficient to show that for every $l=1,...,r-1$, $p=0,...,k_l-1$ and
for $l=r$, $p=0,...,k-2$, we have
\bea\label{eq1p10}
 \sum_{n=\sigma(l,p)}^{N+k}(1-|\lambda_n|^2)
  \frac{d^{k_l-1-p}}{d\mu_l^{k_l-1-p}}
  \left[ \frac{\prod_{m=1}^{n-1}
  (1-\overline{\lambda_m}\lambda_l)} {
  \prod_{j=1,j\neq l}^r  (\mu_l-\mu_j)^{k_j} }\right]
 \left[k_{\lambda_n}B_{n-1}\right]^{(k-1)}(\mu)=0.
\eea
The trick is the same as for simple zeros: redo all the computations from above
for the case $h=1$ for which we will deduce \eqref{eq1p10} from
the interpolating property.

For this, let $\xi\in\T$ and $0<r<1$. Set
\[
 v_{r}(z)=\sum_{n=1}^{N+k}\overline{(A_{\overline{1}}\gamma_n)(r\xi)}\gamma_n(z)
 =\sum_{n=1}^{N+k}\overline{\gamma_n(r\xi)}\gamma_n(z).
\]
Let us first check that $v_r$ interpolates what it should (while this is of course contained
in \cite{AC70} we add here a proof for completeness). To this end, as before,  let $k_{r\xi}$ be
the $H^2$ reproducing kernel at $r\xi\in\D$. Also let $P_{B_{N+k}}$ be the orthogonal
projection onto the space $K_{B_{N+k}}$. Using the Takenaka-Malmquist-Walsh
functions we obtain
\[
 P_{B_{N+k}}k_{r\xi}=\sum_{n=0}^{N+k}\langle k_{r\xi},\gamma_n\rangle\gamma_n
 =\sum_{n=0}^{N+k}\overline{\gamma_n(r\xi)}\gamma_n=v_r.
\]
Hence 
\[
 v_r-k_{r\xi}\in \ker P_{B_{N+k}}=B_{N+k}H^2.
\]
Now all these functions are rational functions with no poles in $\D$ so that we can
pass to the limit $r \to 1^{-}$ to obtain for $\mu$ and $1\le k\le k_0$ (same meaning 
of these parameters as in the first part of the proof),
\[
 v^{(k-1)}(\mu)= k_{\xi}^{(k-1)}(\mu)=\frac{d^{k-1}}{d\mu^{k-1}}k_{\xi}(\mu).
\]
(Note that again the difference $v-k_{\xi}$ is not in $H^2$ since $k_{\xi}$ is not.)

Exactly as in \eqref{eq1p9a} we obtain
\beqa
 \lefteqn{v^{(k-1)}(\mu)=\sum_{n=1}^{N+k}\overline{(A_{\overline{1}}\gamma_n)(\xi)}
 \gamma_n^{(k-1)}(\mu)}\\
 &&=\sum_{l=1}^r\frac{1}{(k_l-1)!}
 \sum_{p=0}^{k_l-1}\binom{k_l-1}{p} 
 \frac{d^{p}}{d\mu_l^{p}}
 \left[\frac{1}{1-\overline{\xi}{\mu_l}  } \right]\times\\
 && \hspace{2cm}\times\sum_{n=\sigma(l,p)}^{N+k}(1-|\lambda_n|^2)
  \frac{d^{k_l-1-p}}{d\mu_l^{k_l-1-p}}
  \left[ \frac{\prod_{m=1}^{n-1}
  (1-\overline{\lambda_m}\lambda_l)} {
  \prod_{j=1,j\neq l}^r  (\mu_l-\mu_j)^{k_j} }\right]
 \left[k_{\lambda_n}B_{n-1}\right]^{(k-1)}(\mu).
\eeqa
The leading coefficient for $l=r$ and $p=k_r-1=k-1$ has already been computed
in \eqref{eq2p10} to be 1. Hence subtracting the term corresponding
to the leading coefficient, we obtain (splitting the sum into the terms for $l\in\{1,2,\ldots,r-1\}$
and $l=r$)
\beqa
\lefteqn{0=\sum_{l=1}^{r-1}\frac{1}{(k_l-1)!}
 \sum_{p=0}^{k_l-1}\binom{k_l-1}{p} 
 \frac{d^{p}}{d\mu_l^{p}}
 \left[\frac{1}{1-\overline{\xi}{\mu_l}  } \right]\times}\\
 &&\hspace{2cm}\times\sum_{n=\sigma(l,p)}^{N+k}(1-|\lambda_n|^2)
  \frac{d^{k_l-1-p}}{d\mu_l^{k_l-1-p}}
  \left[ \frac{\prod_{m=1}^{n-1}
  (1-\overline{\lambda_m}\mu_l)} {
  \prod_{j=1,j\neq l}^r  (\mu_l-\mu_j)^{k_j} }\right]
 \left[k_{\lambda_n}B_{n-1}\right]^{(k-1)}(\mu)\\
 &&+\frac{1}{(k-1)!}
 \sum_{p=0}^{k-2}\binom{k-1}{p} 
 \frac{d^{p}}{d\mu^{p}}
 \left[\frac{1}{1-\overline{\xi}{\mu}  } \right]\times\\
 && \hspace{2cm}\times\sum_{n=\sigma(l,p)}^{N+k}(1-|\lambda_n|^2)
  \frac{d^{k_l-1-p}}{d\mu_l^{k_l-1-p}}
  \left[ \frac{\prod_{m=1}^{n-1}
  (1-\overline{\lambda_m}\mu_l)} {
  \prod_{j=1}^{r-1}  (\mu-\mu_j)^{k_j} }\right]
 \left[k_{\lambda_n}B_{n-1}\right]^{(k-1)}(\mu).\\
\eeqa
The above formula is valid for every $\xi\in \T$. Now, observe that the functions 
\[
 \xi\lmto \frac{d^{p}}{d\mu_l^{p}}
 \left[\frac{1}{1-\overline{\xi}{\mu_l}  } \right]
\]
form a linearly independant family for $l=1,...,r-1$, $p=0,...,k_l-1$ and
for $l=r$, $p=0,...,k-2$, implying that the coefficients
\[
\sum_{n=\sigma(l,p)}^{N+k}
  \frac{d^{k_l-1-p}}{d\mu_l^{k_l-1-p}}
  \left[ \frac{\prod_{m=1}^{n-1}
  (1-\overline{\lambda_m}\mu_l)} {
  \prod_{j=1}^{r-1}  (\mu-\mu_j)^{k_j} }\right]
 \left[k_{\lambda_n}B_{n-1}\right]^{(k-1)}(\mu)
\]
have to vanish in the required ranges of the parameters of $l,p$.
\end{proof}

\begin{remark} \label{remark-norm}
From the identity $k^{h}_{\lambda} = P_I (h k^{I}_{\lambda})$ from \eqref{klh}, and the fact that $\{\gamma_n: n \in \N\}$ forms an orthonormal basis for $K_I$, we see that
\beqa
\|k^{h}_{\lambda}\|^2 &=& \|P_{I} (h k^{I}_{\lambda})\|^2
  = \sum_{n \geq 1} |\langle P_I (h k^{I}_{\lambda}), \gamma_n \rangle|^2
 = \sum_{\geq 1} |\langle k^{I}_{\lambda}, \overline{h} \gamma_n \rangle|^2\\
&=& \sum_{\geq 1} |(A_{\overline{h}} \gamma_n)(\lambda)|^2.
\eeqa
From \eqref{qqq-1} it follows that 
$$\|k^{h}_{\zeta}\|^2 = \sum_{n \geq 1} |(A_{\overline{h}} \gamma_{n})(\zeta)|^2.$$
\end{remark}

\begin{proof}[Proof of Theorem \ref{MT-two}]
Instead of deriving this result from our main theorem, the idea is to use directly
the interpolation property as in \cite{AC70}. The existence of the boundary limits is as before
equivalent to the existence of the function $k^h_{\zeta}$. And from the interpolation
condition \eqref{eq2p27}  the existence of the function $k^h_{\zeta}$ is equivalent
to the solution of the problem
\[
 k^h_{\zeta}(\lambda_n)=\frac{h(\lambda_n)}{1-\overline{\zeta}\lambda_n},
 \quad n\ge 1.
\]
From Shapiro-Shields \cite{SS}, this is equivalent to 
\[
 \sum_{n\ge 1}(1-|\lambda|^2)\left|\frac{h(\lambda_n)}{\zeta-\lambda_n}\right|^2<\infty.
\]
This proves the result.
\end{proof}

\section{Some examples}\label{examples}

Recall that a truncated Toeplitz operators with co-analytic symbol $\overline{h}$ is just 
the restriction of the regular Toeplitz operator with symbol $\overline{h}$ to a model space $K_{I}$. 
Let us discuss some simple examples which illustrate the smoothing effects of applying Toeplitz
operators to functions in $K_I$ or even $H^2$.

\begin{example}\label{firstexample}
The following general fact is well known for functions $f$ in $H^2$ :
\bea\label{estimf}
 |f(z)|  = O\left(\frac{1}{\sqrt{1 - |z|}}\right).
\eea 
It can actually be shown that this growth condition can be replaced in a non-tangential
approach region by a little-oh
condition:
\bea\label{estimflo}
 |f(\lambda)|=o\left(\frac{1}{\sqrt{1-|\lambda|}}\right),\quad \lambda\stackrel{\angle}{\lra}\zeta\in\T.
\eea
The notation $z\stackrel{\angle}{\lra}\zeta$ means that $z$ tends non tangentially to $\zeta$.
As a consequence, we observe that if $\vp$ is anaytic and $|\vp(z)|\le \sqrt{1-|z|}$ as 
$z\stackrel{\angle}{\lra}\zeta$ --- which is for instance the case when $\vp(z)=\sqrt{\zeta-z}$ ---
then every function in the range of the {\it analytic} Toeplitz operator
$T_{\vp}f$, where actually $f\in H^2$, will have a boundary limit
(zero) at $\zeta$.
\end{example}

\begin{example}
 The situation is more intricate 
when considering co-analytic symbols.
A simple observation is the following: If $h(z)={1-z}$, then for every function
$f\in H^2$,
\[
 T_{\overline{h}}f(z)=T_{\overline{1-z}}f(z)=f(z)-\frac{f(z)-f(0)}{z}=\frac{f(z)(z-1)}{z}-\frac{f(0)}{z}
\]
which tends in fact to $-f(0)$ (and which is in general not 0)
when $z\stackrel{\angle}{\lra}1$ (we have used \eqref{estimf}).

\end{example}

\begin{example} \label{ex4.1}
In this example we use Theorem \ref{MT-two} to show that the
natural multiplier $\sqrt{1-z}$, which makes every function in $H^2$ vanish
non-tangentially at 1 (as observed in Example \ref{firstexample}),
is not sufficient for co-analytic Toeplitz operators.

Let $\Lambda=(1-\frac{1}{2^n})_{n\ge 1}$, $I$ be the Blaschke product with these zeros, and $h_{\eps}(z)=(1-z)^{1/2+\eps}$. Then
every function $f\in \Rng A_{\overline{h_{\eps}}}$ has non-tangential limit in 1, if and only if the
condition in \eqref{eq1p29} is fulfilled. Now observe that
\[
 \sum_{n\ge 1}(1-|\lambda_n|^2)\left|\frac{h_{\eps}(\lambda_n)}{1-\lambda_n}\right|^2
 \simeq \sum_{n\ge 1}2^{-n}\left|\frac{1/2^{n(1/2+\eps)}}{1/2^n}\right|^2
 =\sum_{n\ge 1}2^{-2n\eps}
\] 
which converges if and only if $\eps>0$. So we need the symbol $h$ to decrease faster than $\sqrt{1-z}$ to ensure existence of the boundary limits of functions in $\Rng A_{\overline{h}}$.

It is possible to consider a decrease closer to $\sqrt{1-z}$, for
example $$h(z)=\log 2 \frac{\sqrt{1-z}}{(-\log(1-z))}$$ (for which $h(\lambda_n)=1/(n2^{n/2})$),
but we can never reach $\sqrt{1-z}$.
\end{example}

\begin{example}\label{example4.2}
In this next proposition, we see  that Theorem \ref{MT-two} is not true for non-interpolating Blaschke sequences.
\end{example}

\begin{proposition}\label{prop4.1}
There exists a point $\zeta \in \T$,  a Blaschke product $I$ whose zeros $\Lambda\subset \D$ satisfy the condition \eqref{eq1p29} at $\zeta$, a function $h\in K_I$ such that $A_{\overline{h}}$ is bounded
on $K_I$. 
Still there are functions in $A_{\overline{h}}K_I$ which do not have finite non-tangential boundary limits 
at  $\zeta$.
\end{proposition}

\begin{proof}
The proof of this result relies on a result concerning interpolation on finite unions
of interpolating sequences \cite{BNO, AH96}.
Let $\Lambda_1=\{\lambda_n^1\}_{n\ge 1}=\{1-1/2^n\}_{n\ge 1}$, 
which is 
an interpolating sequence \cite{Garnett} and let $\Lambda_2=\{\lambda_n^2\}_{n\ge 1}$ 
satisfy $|b_{\lambda_n^1}(\lambda_n^2)|= 1/n$. The sequence $\Lambda_2$ is a sufficiently small
perturbation of $\Lambda_1$ such that $\Lambda_2$ will also be interpolating. Also note that
the sequence $\Lambda:=\Lambda_1\cup\Lambda_2$ accumulates
non tangentially at $\zeta=1$.
Let
 $$v_n^1=\frac{1}{n2^{n/2}}\ , \quad v_n^2=0.$$
 The central result used here is the following:
a sequence of values $(w_n^k)_{n\ge 1;k=1,2}$ is a trace of a function $f\in H^2$
(or $K_I$) if and only if \cite{BNO, AH96} 
\bea\label{eq3p29}
 \sum_{n\ge 1}(1-|\lambda_n^1|^2)\left[|w_n^1|^2+\left|\frac{w_n^2-w_n^1}{b_{\lambda_n^1}
 (\lambda_n^2)}\right|^2\right]<\infty.
\eea

For the values $w_n^i=v_n^i$ that we have given above, we get:
\[
  \sum_{n\ge 1}(1-|\lambda_n^1|^2)\left[|v_n^1|^2+\left|\frac{v_n^2-v_n^1}{b_{\lambda_n^1}
 (\lambda_n^2)}\right|^2\right]
 \simeq\sum_{n\ge 1}\frac{1}{2^n}\left[\left(\frac{1}{n2^{n/2}}\right)^2
 +\left|\frac{1/(n2^{n/2})}{1/n}\right|^2\right]<\infty,
\]
and so that there is a function $h$ in $H^2$ or in $K_I$ taking the values $v_n^i$ at
$\lambda_n^i$.

Next we check the condition \eqref{eq1p29}. Note that since $h(\lambda_n^2)=v_n^2=0$,
we only have to sum over $\Lambda_1$. Indeed we get
\beqa 
 \sum_{n\ge 1}(1-|\lambda_n^1|^2)\left|\frac{h(\lambda_n^1)}{1-\lambda_n^1}\right|^2
 \simeq \sum_{n\ge 1}\frac{1}{2^n}\left|\frac{1/(n2^{n/2})}{1/2^n}\right|^2
 =\sum_{n\ge 1}\frac{1}{n^2}<\infty.
\eeqa

Let us check that the sequence defined by 
$$w_n^i:= \frac{h(v_n^i)}{\zeta-\lambda_n^i}, \quad n\ge 1,i=1,2,$$ cannot be realized
by a function in $K_I$ so that $k^h_{\zeta}$ does not exist and hence there are
functions in $A_{\overline{h}}K_I$ that do not admit boundary limits in $\zeta=1$.
In order to do so, we have to check that this sequence does not satisfy the condition
\eqref{eq3p29}. We compute to get 
\beqa
 \sum_{n\ge 1}(1-|\lambda_n^1|^2)\left[|w_n^1|^2+\left|\frac{w_n^2-w_n^1}{b_{\lambda_n^1}
 (\lambda_n^2)}\right|^2\right]
 &\simeq&\sum_{n\ge 1}\frac{1}{2^n}\left(\left|\frac{1/(n2^{n/2})}{1/2^n}\right|^2
 +\left|\frac{\frac{1/(n2^{n/2})}{1/2^n}-0}{1/n}\right|^2\right)\\
 &=&\sum_{n\ge 1}(\frac{1}{n^2}+1)=+\infty
\eeqa
so that this value sequence cannot be realized by a function in $K_I$.

We finally have to check that $A_{\overline{h}}$ is bounded on $K_I$. For this, note
that $K_I$ is an $l^2$-sum of $K_{B_n}$ where $B_n$ is the finite Blaschke product with zeros $\{\lambda_n^1,\lambda_n^2\}$ (see \cite[Theorem C3.2.14]{Nik2}). 
By this we mean that every $f \in K_I$ can be written as 
$$f = \sum_{n \geq 1} f_n, \quad f_n \in K_{B_n}, \quad \|f\|^2 \asymp \sum_{n \geq 1}  \|f_n\|^2.$$

We use the Takenaka-Malmquist-Walsh system to generate $K_{B_n}$:
\[
 \gamma_{n,1}(z)=\frac{\sqrt{1-|\lambda_n^1|^2}}{1-\overline{\lambda_n^1}z},
 \quad \gamma_{n,2}(z)=\frac{\sqrt{1-|\lambda_n^2|^2}}{1-\overline{\lambda_n^2}z}
 \frac{z-\lambda_n^1}{1-\overline{\lambda_n^1}z}.
\]
So, every function $f\in K_I$ can be written as
\[
 f=\sum_{n\ge 1}(\alpha_{n,1}\gamma_{n,1}+\alpha_{n,2}\gamma_{n,2})
\]
with $\|f\|^2\simeq\sum_{n\ge 1}|\alpha_{n,1}|^2+|\alpha_{n,2}|^2<\infty$.
Apply now $A_{\overline{h}}$ to this sum (we could start with finite sums and
check that we have a uniform norm control).
Clearly
\[
 A_{\overline{h}}\gamma_{n,1}=\overline{h(\lambda_n^1)}\gamma_{n,1}.
\]
The action of $A_{\overline{h}}\gamma_{n,2}$ can be deduced from Corollary \ref{prop6.3}.
We obtain
\[
 (A_{\overline{h}}\gamma_{n,2})(z)=\sqrt{1-|\lambda_n^2|^2}\left[
 \frac{\overline{h(\lambda_n^1)}}{1-\overline{\lambda_n^1}z}
  \frac{1}{\overline{b_{\lambda_n^2}(\lambda_n^1)}}\frac{(1-|\lambda_n^1|^2)}
  {1-\overline{\lambda_n^1}\lambda_n^2}  
+\frac{\overline{h(\lambda_n^2)}}{1-\overline{\lambda_n^1}z}
 \frac{1}{\overline{b_{\lambda_n^1}(\lambda_n^2)}}\right].
\]
Note that $h(\lambda_n^2)=v_n^2=0$, and hence
\[
 (A_{\overline{h}}\gamma_{n,2})(z)=
 \sqrt{1-|\lambda_n^2|^2}
 \frac{\overline{h(\lambda_n^1)}}{1-\overline{\lambda_n^1}z}
  \frac{1}{\overline{b_{\lambda_n^2}(\lambda_n^1)}}\frac{(1-|\lambda_n^1|^2)}
  {1-\overline{\lambda_n^1}\lambda_n^2}  = \beta_n\gamma_{n,1},
\]
where
\[
 \beta_n=\frac{\sqrt{1-|\lambda_n^2|^2}\sqrt{1-|\lambda_n^1|^2}}
  {1-\overline{\lambda_n^1}\lambda_n^2}
 \frac{\overline{h(\lambda_n^1)}}{\overline{b_{\lambda_n^2}(\lambda_n^1)}} .
\]
In view of the explicit values of $\lambda_n^i$, $h(\lambda_n^i)$ and 
$|b_{\lambda_n^2}(\lambda_n^1)|$,
the sequence $\{\beta_n\}_{n \geq 1}$ is bounded (it actually tends to zero quickly).
We thus get
\[
 A_{\overline{h}}f=\sum_{n\ge 1}\left(\alpha_{n,1}\overline{h(\lambda_n^1)}\gamma_{n,1}
 +\alpha_{n,2}\beta_n\gamma_{n,1}\right)
 =\sum_{n\ge 1}\left(\alpha_{n,1}\overline{h(\lambda_n^1)}
 +\alpha_{n,2}\beta_n\right)\gamma_{n,1}
\]
and hence, since $h$ is also bounded on $\Lambda_1$ (actually decreasing
very fast to 0),
\[
 \|A_{\overline{h}}f\|^2=\sum_{n\ge 1}|\overline{h(\lambda_n^1)}
 \alpha_{n,1}+\beta_n\alpha_{n,2}|^2
 \lesssim \sum_{n\ge 1}|\alpha_{n,1}|^2+|\alpha_{n,2}|^2\simeq \|f\|^2.
\]
\end{proof}

Note how in this example
$B$ does not have non-tangential limit at $\zeta=1$. Indeed, $B$ vanishes at its zeros and, in the
middle
between two successive pairs $\{\lambda_n^1,\lambda_n^2\}$ and $\{\lambda_{n+1}^1,
\lambda_{n+1}^2\}$, we are far from the elements of the two interpolating sequences
$\Lambda_1$ and $\Lambda_2$. Thus $B$ will be big at these points.

The second remark is that for $h(z)=1-z$ we have already seen that every function in
$\Rng A_{\overline{1-z}}$ will have a limit at $\zeta=1$. Choosing, as mentioned in Example
\ref{ex4.1},
$h(z)=\log 2\sqrt{1-z}/(-\log(1-z))$ (which gives exactly $h(\lambda_n^1)=1/(n2^{n/2})$)
and $w_n^i=h(\lambda_n^i)$, it can be checked that \eqref{eq3p29} is true so that for this
function $h$, every $f\in A_{\overline{h}}K_I$ has non tangential limit at $\zeta=1$.
If the reader prefers a function in $K_I$, it is sufficient to project $h$ into $K_I$ which
does not change the values on $\Lambda$.

The arguments given in the proof of Proposition \ref{prop4.1} indicate
how to adapt the construction to generalize Theorem \ref{MT-two} to
finite unions of interpolating sequences.

\begin{example}\label{example4.3}
In this final example, we apply Theorem \ref{MT-one} to a 
sequence which is not a finite union of interpolating sequences. Fix $\beta\in(1/2,1)$.
Let 
\[
 \lambda_n=1-\frac{1}{2^{n^{\beta}}},
\]
and let $B$ be the Blaschke product associated with
the sequence $\Lambda=\{\lambda_n\}_n$. 
Since the convergence of this sequence to $1$ is sub-exponential, there will be dyadic
intervals $[1-1/2^{n},1-1/2^{n+1}]$ in the radius $[0,1)$ containing arbitrarily big
numbers of elements of $\Lambda$ so that the associated measure 
$\sum_{n\ge 1}(1-|\lambda_n|^2)\delta_{\lambda_n}$ cannot be Carleson (see
\cite{Garnett} for more information on Carleson measures).

Let us first estimate $|B_{\lambda_n}(\lambda_n)|$ where $B_{\lambda_n}$ is the
Blaschke product associated with the sequence $\Lambda\setminus\{\lambda_n\}$.
In order to do these estimates, we will consider
\beqa
 \log|B_{\lambda_n}(\lambda_n)|^{-1}&=&\sum_{k\neq n} \log|b_{\lambda_k}(\lambda_n)|^{-1}
  =\sum_{k\neq n}\log\left|\frac{1/2^{n^{\beta}}-1/2^{k^{\beta}}}{1/2^{n^\beta}+1/2^{k^\beta}
 -1/2^{k^\beta+n^{\beta}}}\right|^{-1}\\
 &=&\sum_{k\neq n}\log\left|\frac{2^{n^{\beta}}+2^{k^{\beta}}-1}{2^{n^\beta}-2^{k^\beta}
 }\right|.
\eeqa
We can suppose that $n$ is large enough so that we do not have to worry about the $-1$
which occurs in the last numerator. 
We will now split the summation (in the index $k$) into 4 (or 2) pieces.

\underline{Case 1:} Consider $n+1\le k\le n+n^{1-\beta}$.
Then
\[
 \left|\frac{2^{n^{\beta}}+2^{k^{\beta}}}{2^{n^\beta}-2^{k^\beta}}\right|
 \simeq\left|\frac{2^{k^{\beta}}}{2^{n^\beta}-2^{k^\beta}}\right|
 =\frac{2^{k^{\beta}}}   {2^{k^{\beta}}-2^{n^{\beta}}}
 =\frac{1}{1-2^{n^{\beta}-k^{\beta}}}.
\]
Note that
\bea\label{eq1p31}
 0&\ge& n^{\beta}-k^{\beta}\ge n^{\beta}-(n+n^{1-\beta})^{\beta}
 =n^{\beta}-n^{\beta}(1+n^{-\beta})^{\beta}\nn\\
 &=&n^{\beta}-n^{\beta}(1+\beta/n^{\beta}+o(1/n^{\beta}))\nn\\
 &=&-\beta+o(1).
\eea
So, $-1<-\ln 2<-\beta\ln 2\lesssim (\ln 2)(n^{\beta}-k^{\beta})\le 0$ (where the
``$\lesssim$" is asymptotically, for $n\to \infty$, a ``$\le$"), and hence
\[
 2^{n^{\beta}-k^{\beta}}=e^{(\ln 2)(n^{\beta}-k^{\beta})}
 \simeq 1+(\ln 2)(n^{\beta}-k^{\beta}), 
\]
so that
\[
 \left|\frac{2^{n^{\beta}}+2^{k^{\beta}}}{2^{n^\beta}-2^{k^\beta}}\right|
 \simeq \frac{1}{1-2^{n^{\beta}-k^{\beta}}}\simeq \frac{1}{\ln 2(k^{\beta}-n^{\beta})}
\]
Now, setting $k=n+l$ with $l\in \{1,2,\ldots,n^{1-\eps}\}$ we get
\[
 (n+l)^{\beta}-n^{\beta}=n^{\beta}(1+\frac{l}{n})^{\beta}-n^{\beta}
 \simeq \frac{\beta l}{n^{1-\beta}}.
\]
Hence
\[
 \sum_{k=n+1}^{n+n^{1-\beta}}\log|b_{\lambda_k}(\lambda_n)|^{-1}
 \simeq \sum_{l=1}^{n^{1-\beta}}\log \frac{n^{1-{\beta}}}{\beta l\ln 2}.
\]
And switching back to the product we get
\[
 \prod_{k=n+1}^{n+n^{1-\beta}}|b_{\lambda_k}(\lambda_n)|^{-1}
 \simeq \left(\frac{n^{1-\beta}}{\beta\ln 2}\right)^{n^{1-\beta}}\frac{1}{(n^{1-\beta})!}.
\]
Using Stirling's formula
\[
 \frac{N^N}{N!}\simeq \frac{e^N}{\sqrt{2\pi N}},
\]
we obtain with $N=n^{1-\beta}$,
\bea\label{p32c1}
  \prod_{k=n+1}^{n+n^{1-\beta}}|b_{\lambda_k}(\lambda_n)|^{-1}
 \simeq \left(\frac{e}{\beta\ln 2}\right)^{n^{1-\beta}}\frac{1}{\sqrt{2\pi n^{1-\beta}}}
 \lesssim e^{cn^{1-\beta}}
\eea
for some suitable constant $c$.

\underline{Case 2:} Suppose now that $k\ge n+n^{1-\beta}$.
Observe
\[
 \left|\frac{2^{n^{\beta}}+2^{k^{\beta}}}{2^{n^{\beta}}-2^{k^{\beta}}}\right|
 =\left| 1+2\frac{2^{n^{\beta}}}{2^{k^{\beta}}-2^{n^{\beta}}}\right|.
\]
Then
\[
  \frac{2^{n^{\beta}}}{2^{k^{\beta}}-2^{n^{\beta}}}
 \le  \frac{2^{n^{\beta}}}{2^{(n+n^{1-\beta})^{\beta}}-2^{n^{\beta}}}
 =\frac{1}{2^{(n+n^{1-\beta})^{\beta}-n^{\beta}}-1},
\]
which, by similar computations as in \eqref{eq1p31}, is controlled by 
\[
 \frac{1}{2^{\beta}-1}.
\]
This enables us now to write
\[
 \log\left| 1+2\frac{2^{n^{\beta}}}{2^{k^{\beta}}-2^{n^{\beta}}}\right|
 \simeq\frac{2}{2^{k^{\beta}-n^{\beta}}-1}\lesssim \frac{2^{n^{\beta}}}{2^{k^{\beta}}}.
\]
Using the estimate
\[
 \int_{M}^{\infty}e^{-x^{\beta}}dx\simeq M^{1-\beta}e^{-M^{\beta}},
\]
we can compute
\[
 \sum_{k\ge n+n^{1-\beta}}\log\left| 1+2\frac{2^{n^{\beta}}}{2^{k^{\beta}}-2^{n^{\beta}}}\right|
 \lesssim 2^{n^{\beta}}\sum_{k\ge n+n^{1-\beta}}\frac{1}{2^{k^{\beta}}}
 \simeq 2^{n^{\beta}}(n+n^{1-\beta})^{1-\beta}\frac{1}{2^{(n+n^{1-\beta})^{\beta}}}
 \simeq n^{1-\beta}
\]
so that we also get
\[
 \prod_{k\ge n+n^{1-\beta}}|b_{\lambda_k}(\lambda_n)|^{-1}
 \le e^{cn^{1-\beta}}
\] 
for some suitable constant $c$.
\\

We will also include a brief discussion of the 
cases 3 --- $(n-n^{1-\eps})\le k\le n-1$ --- and 4 --- $1\le k\le (n-n^{1-\eps})$ ---
which are treated in essentially the same way.

\underline{Case 3:} Consider $n-n^{1-\beta}\le k\le n-1$.
Then
\[
 \left|\frac{2^{n^{\beta}}+2^{k^{\beta}}}{2^{n^\beta}-2^{k^\beta}}\right|
 \simeq\left|\frac{2^{n^{\beta}}}{2^{n^\beta}-2^{k^\beta}}\right|
 =\frac{1}{1-2^{k^{\beta}-n^{\beta}}}.
\]
Now 
\[
 0\ge k^{\beta}-n^{\beta}\ge (n-n^{1-\beta})^{\beta}-n^{\beta}
 \simeq-\beta+o(1)
\]
as in \eqref{eq1p31}.
So, $-1<-\ln 2<-\beta\ln 2\lesssim (\ln 2)(k^{\beta}-n^{\beta})\le 0$ (where the
``$\lesssim$" is asymptotically, for $n\to \infty$, a ``$\le$"), and 
we can conclude as in the case 1 to obtain 
\beqa
  \prod_{k=n-n^{1-\beta}}^{n-1}|b_{\lambda_k}(\lambda_n)|^{-1}
 \simeq \left(\frac{e}{\beta\ln 2}\right)^{n^{1-\beta}}\frac{1}{\sqrt{2\pi n^{1-\beta}}}
 \lesssim e^{cn^{1-\beta}}
\eeqa
for some suitable constant $c$.

\underline{Case 4:} Suppose now that $k\le n-n^{1-\beta}$.
Observe
\[
 \left|\frac{2^{n^{\beta}}+2^{k^{\beta}}}{2^{n^{\beta}}-2^{k^{\beta}}}\right|
 =\left| 1+2\frac{2^{k^{\beta}}}{2^{n^{\beta}}-2^{k^{\beta}}}\right|.
\]
Then
\[
  \frac{2^{k^{\beta}}}{2^{n^{\beta}}-2^{k^{\beta}}}
 =\frac{1}{2^{n^{\beta}-k^{\beta}}-1}
  \le  \frac{1}{2^{n^{\beta}-(n-n^{1-\beta})^{\beta}}-1}
\]
which, by similar computations as in \eqref{eq1p31}, is controlled by 
\[
 \frac{1}{2^{\beta}-1}.
\]
This enables us now to write
\[
 \log\left| 1+2\frac{2^{k^{\beta}}}{2^{n^{\beta}}-2^{k^{\beta}}}\right|
 \simeq2 \frac{2^{k^{\beta}}}{2^{n^{\beta}}-2^{k^{\beta}}}
 \lesssim \frac{2^{k^{\beta}}}{2^{n^{\beta}}}.
\]
Using
\[
 \int_1^{M}e^{x^{\beta}}dx\simeq M^{1-\beta}e^{M^{\beta}}
\]
we can estimate
\[
 \sum_{k\le n-n^{1-\beta}}\log\left| 1+2\frac{2^{n^{\beta}}}{2^{k^{\beta}}-2^{n^{\beta}}}\right|
 \lesssim \frac{1}{2^{n^{\beta}}}\sum_{k\le n-n^{1-\beta}}{2^{k^{\beta}}}
 \simeq \frac{1}{2^{n^{\beta}}}(n-n^{1-\beta})^{1-\beta}{2^{(n-n^{1-\beta})^{\beta}}}
 \simeq n^{1-\beta}
\]
so that we also get
\[
 \prod_{k\le n-n^{1-\beta}}|b_{\lambda_k}(\lambda_n)|^{-1}
 \le e^{cn^{1-\beta}}
\] 
for some suitable constant $c$.
\\

Putting this all together we obtain
\[
 \delta_n:=|B_{\lambda_n}(\lambda_n)|\ge e^{-cn^{1-\beta}}
\]
for some suitable constant $c$.
\\

Let us now return to our problem of estimating $|A_{\overline{h}}\gamma_n(\zeta)|$.
From Proposition \ref{prop6.3}, we have
\[
 (A_{\overline{h}}\gamma_n)(\zeta)=\sqrt{1-|\lambda_n|^2}
  \sum_{l=1}^n \frac{\overline{h(\lambda_l)}}   
 {1-\overline{\lambda_l}{\zeta}}\frac{1}{\overline{(B_n)_{\lambda_l}(\lambda_l)}}
 \frac{1-|\lambda_l|^2}{1-\overline{\lambda_l}\lambda_n}
\]
Our Blaschke product constructed above accumulates at $\zeta=1$ and contains
only points in $(0,1)$. Let $h(z)=(1-z)^{1-\eps}$. Then
\[
 |(A_{\overline{h}}\gamma_n)(\zeta)|
 \lesssim\sqrt{1-\lambda_n}\sum_{l=1}^n\frac{(1-\lambda_l)^{1-\eps}}{1-\lambda_l}
 \frac{1}{|(B_n){\lambda_l}(\lambda_l)|}\frac{1-\lambda_l}{1-\lambda_l}
 \lesssim\sqrt{1-\lambda_n}\sum_{l=1}^n\frac{(1-\lambda_l)^{-\eps}}{\delta_l}
\]
Recall that $\lambda_n=1-\frac{1}{2^{n^\beta}}$ and $\delta_n\ge e^{-cn^{1-\beta}}$. Hence
\[
 |(A_{\overline{h}}\gamma_n)(\zeta)|
 \lesssim\frac{1}{2^{n^{\beta}/2}}\sum_{l=1}^n\frac{2^{\eps l^{\beta}}}{e^{-cl^{1-\beta}}}
 \lesssim \frac{1}{2^{n^{\beta}/2}}n{2^{\eps n^{\beta}}}{e^{cn^{1-\beta}}}
 =n2^{\eps n^{\beta}+c\ln 2 n^{1-\beta}-n^{\beta}/2}
\]
which is square summable as long as $\eps<1/2$ and $\beta>1/2$.

Note that again our zeros are contained in the radius $(0,1)$ and the function $h$ has to go
slightly faster to zero than the square root as in the situation when $\Lambda$ was an
interpolating Blaschke sequence in $(0,1)$.

Note also that in this example 
$$\angle \lim_{\lambda \to 1} B(\lambda) = 0.$$

\end{example}

\section{Unbounded operators}\label{S5}

For any $h \in H^2$ the truncated Toeplitz operator $A_{\overline{h}}$ turns out to be a closed,
densely defined operator on $K_I$ with a domain $\mathcal{D}(A_{\overline{h}})$ which contains $K_I \cap H^{\infty}$ \cite{Sa08}. If one looks closely at the proof of the two main theorems of this paper (Theorem \ref{MT-one} and Theorem \ref{MT-two}), one realizes that the 
sufficiency parts 
still hold but with $\Rng A_{\overline{h}}$ defined as $A_{\overline{h}} \mathcal{D}(A_{\overline{h}})$. 

Furthermore, 
the conditions given in these theorems are still sufficient for every $h \in H^2$ when $\Rng A_{\overline{h}}$, as defined in the previous paragraph, 
is replaced by the linear manifold $\{f_h: f \in K_I\}$, where 
$f_h$ is defined by the left-hand side of \eqref{formally}, i.e., 
$$
 f_h(\lambda):=\langle f, (\Id - \overline{\lambda} A_z)^{-1} P_I h \rangle.
$$
Repeating the argument in \eqref{formally0}, we can also write
\[
 f_h(\lambda)=\langle f,k_{\lambda}P_Ih \rangle.
\]
Note that the linear manifold $\{f_h: f \in K_I\}$ is not necessarily a subset of $K_I$. However,  
$$
f_{h}(\lambda) = \int_{\T} \frac{f(\xi)\overline{(P_Ih)(\xi)}}{1 - \overline{\xi} \lambda} dm(\xi) 
$$ 
is a Cauchy transform of the finite measure $f \overline{P_I h} dm$. Since Cauchy transforms of finite measures on the circle are known to belong to all the Hardy classes $H^p$ for $0 < p < 1$ \cite[p.~43]{CMR06}, we know that the non-tangential limits of $f_h$ exist almost everywhere. Theorems \ref{MT-one} and \ref{MT-two} give sufficient conditions when these non-tangential limits exist at specific points of the circle. 

\section{Open questions}\label{S6}

Conspicuously missing from this paper is a discussion of what happens to $\Rng A_{\overline{h}}$ when $I$ is a general inner function $I = B s_{\mu}$ and not necessarily a Blaschke product as was discussed here. In this case, if we are aiming for a similar characterization as in Theorems \ref{MT-one} and \ref{MT-two}, we would need a different orthonormal basis than $\{\gamma_n: n \in \N\}$. So suppose that $\{\vp_n: n \in \N\}$ is an orthonormal basis for $K_I$ for a general inner function $I$. Some examples can be found in  \cite{AC70a}. Then Proposition \ref{AC-Th} still holds and so the non-tangential boundary values at a fixed point $\zeta \in \T$ will exist for all functions from $\Rng A_{\overline{h}}$ if and only if the kernel functions $k^{h}_{\lambda}$ remain bounded whenever $\lambda \to \zeta$ non-tangentially. The exact same computation as in Remark \ref{remark-norm} will show that 
$$\|k^{h}_{\lambda}\|^2 = \sum_{n \geq 1} |(A_{\overline{h}} \vp_n)(\lambda)|^2.$$
At this point, two problems stand in our way. The first is to prove that 
 $(A_{\overline{h}} \vp_n)(\zeta)$ exists as it did so nicely for $(A_{\overline{h}} \gamma_n)(\zeta)$. Recall that $A_{\overline{h}} \gamma_n$ is a rational function whose poles are off of $\D^{-}$. Is $A_{\overline{h}} \vp_n$ such a nice function so we can compute $(A_{\overline{h}} \vp_n)(\zeta)$ without any difficulty?
The second problem, assuming we can overcome the first, is to show that perhaps the natural choice of kernel function 
$$k:= \sum_{n \geq 1}  \overline{(A_{\overline{h}}\vp_n)(\zeta)} \vp_n$$
satisfies the interpolation condition in Corollary \ref{C-int-prob}.

One could also ask whether or not one could extend our results to determine, as in Ahern-Clark, when the \emph{derivatives} (of certain orders)  of functions in $\Rng A_{\overline{h}}$, have non-tangential limits at $\zeta \in \T$. 

\bibliography{range}

\end{document}